\documentclass[a4paper,UKenglish,cleveref, autoref, thm-restate]{lipics-v2021}

\pdfoutput=1 
\hideLIPIcs  
\nolinenumbers

\usepackage{pgfplotstable}
\usepackage{subcaption}
\usetikzlibrary{calc}
\usepackage{algorithm}
\usepackage[noend]{algpseudocode}

\usepackage{thmtools}
\usepackage{comment}
\usepackage{url}

\newcommand{\N}{\mathbb{N}}
\newcommand{\Z}{\mathbb{Z}}
\newcommand{\FP}{\mathrm{FP}}
\newcommand{\NP}{\mathrm{NP}}
\newcommand{\diff}{\mathrm{diff}}
\newcommand{\dom}{\mathrm{dom}}
\newcommand{\dist}{\mathrm{dist}}
\newcommand{\CC}{\mathcal{C}}
\newcommand{\FF}{\mathcal{F}}
\newcommand{\II}{\mathcal{I}}

\newcommand{\koynnos}{k\"oynn\"os}
\newcommand{\Koynnos}{K\"oynn\"os}
\newcommand{\kynnos}{kynn\"os}
\newcommand{\Kynnos}{Kynn\"os}

\newcommand{\annu}[8]{\begin{bmatrix} #1 & #2 & #3 & #4 \\
#5 & #6 & #7 & #8 \end{bmatrix}}

\newcommand{\sannu}[8]{\left[ \begin{smallmatrix} #1 & #2 & #3 & #4 \\
#5 & #6 & #7 & #8 \end{smallmatrix} \right]}

\newtheorem{question}{Question}

\definecolor{blu}{rgb}{0,1,1}

\makeatletter
\tikzset{
    zero color/.initial=white,
    zero color/.get=\zerocol,
    zero color/.store in=\zerocol,
    one color/.initial=red,
    one color/.get=\onecol,
    one color/.store in=\onecol,
    two color/.initial=blue,
    two color/.get=\twocol,
    two color/.store in=\twocol,
    cell wd/.initial=1,
    cell wd/.get=\cellwd,
    cell wd/.store in=\cellwd,
    cell ht/.initial=1,
    cell ht/.get=\cellht,
    cell ht/.store in=\cellht,
}
\makeatother

\newcommand{\drawit}[1][1]{

      \pgfplotstablegetrowsof{\matrixfile} 
      \pgfmathtruncatemacro{\totrow}{\pgfplotsretval}
      \pgfplotstablegetcolsof{\matrixfile} 
      \pgfmathtruncatemacro{\totcol}{\pgfplotsretval}
      
      \pgfplotstableforeachcolumn\matrixfile\as\col{
        \pgfplotstableforeachcolumnelement{\col}\of\matrixfile\as\colcnt{%
          \ifnum\colcnt=0
          \fill[black]($ -\pgfplotstablerow*(0,\cellht) + \col*(\cellwd,0) $) rectangle +(\cellwd,\cellht);
          \fi
          \ifnum\colcnt=1
          \fill[blu]($ -\pgfplotstablerow*(0,\cellht) + \col*(\cellwd,0) $) rectangle+(\cellwd,\cellht);
          \fi
          \ifnum\colcnt=2
          \fill[white]($ -\pgfplotstablerow*(0,\cellht) + \col*(\cellwd,0) $) rectangle +(\cellwd,\cellht);
          \draw ($ -\pgfplotstablerow*(0,\cellht) + (\cellwd*0.5,\cellht*0.5) + \col*(\cellwd,0) $) circle (0.1);
          \fi
          \ifnum\colcnt=3
          \fill[gray!50!white]($ -\pgfplotstablerow*(0,\cellht) + \col*(\cellwd,0) $) rectangle+(\cellwd,\cellht);
          \draw ($ -\pgfplotstablerow*(0,\cellht) + (\cellwd*0.5,\cellht*0.5) + \col*(\cellwd,0) $) circle (0.1);
          \fi
          \ifnum\colcnt=4
          \fi
        }
      }
      \ifnum#1=1
      \draw[thick,black] (0,-\totrow+1) rectangle (\totcol,1);
      \fi
}

\newcommand{\drawitt}{

      \pgfplotstablegetrowsof{\matrixfile} 
      \pgfmathtruncatemacro{\totrow}{\pgfplotsretval}
      \pgfplotstablegetcolsof{\matrixfile} 
      \pgfmathtruncatemacro{\totcol}{\pgfplotsretval}
      
      \pgfplotstableforeachcolumn\matrixfile\as\col{
        \pgfplotstableforeachcolumnelement{\col}\of\matrixfile\as\colcnt{%
          \ifnum\colcnt=0
          \fill[black]($ -\pgfplotstablerow*(0,\cellht) + \col*(\cellwd,0) $) rectangle +(\cellwd,\cellht);
          \fi
          \ifnum\colcnt=1
          \fill[white]($ -\pgfplotstablerow*(0,\cellht) + \col*(\cellwd,0) $) rectangle+(\cellwd,\cellht);
          \fi
          \ifnum\colcnt=2
          \fill[white]($ -\pgfplotstablerow*(0,\cellht) + \col*(\cellwd,0) $) rectangle +(\cellwd,\cellht);
          \draw ($ -\pgfplotstablerow*(0,\cellht) + (\cellwd*0.5,\cellht*0.5) + \col*(\cellwd,0) $) circle (0.1);
          \fi
          \ifnum\colcnt=3
          \fill[gray!50!white]($ -\pgfplotstablerow*(0,\cellht) + \col*(\cellwd,0) $) rectangle+(\cellwd,\cellht);
          \draw ($ -\pgfplotstablerow*(0,\cellht) + (\cellwd*0.5,\cellht*0.5) + \col*(\cellwd,0) $) circle (0.1);
          \fi
        }
      }
      \draw[black!20!white] (0,-\totrow+1) grid (\totcol,1);
      \draw[thick,black] (0,-\totrow+1) rectangle (\totcol,1);

}

\title{What can oracles teach us about the ultimate fate of life?}

\author{Ville Salo}{Department of Mathematics and Statistics, University of Turku, Finland}{vosalo@utu.fi}{https://orcid.org/0000-0002-2059-194X}{Research supported by Academy of Finland grant 2608073211.} 
\author{Ilkka T\"orm\"a}{Department of Mathematics and Statistics, University of Turku, Finland}{iatorm@utu.fi}{https://orcid.org/0000-0001-5541-8517}{} 

\authorrunning{V. Salo and I. T\"orm\"a}

\Copyright{Ville Salo and Ilkka T\"orm\"a} 

\ccsdesc[500]{Theory of computation~Formal languages and automata theory}

\keywords{Game of Life, cellular automata, limit set, symbolic dynamics} 

\relatedversiondetails{Full Version}{https://arxiv.org/abs/2202.07346} 

\supplementdetails[subcategory={source code},cite=Program]{Software}{https://github.com/ilkka-torma/gol-agars} 

\acknowledgements{We thank the anonymous referees for their useful suggestions, and the nice people of the ConwayLife forum for their interest.}

\EventEditors{Miko{\l}aj Boja\'{n}czyk, Emanuela Merelli, and David P. Woodruff}
\EventNoEds{3}
\EventLongTitle{49th International Colloquium on Automata, Languages, and Programming (ICALP 2022)}
\EventShortTitle{ICALP 2022}
\EventAcronym{ICALP}
\EventYear{2022}
\EventDate{July 4--8, 2022}
\EventLocation{Paris, France}
\EventLogo{}
\SeriesVolume{229}
\ArticleNo{119}

\begin{document}
\maketitle

\begin{abstract}
We settle two long-standing open problems about Conway's Life, a two-dimensional cellular automaton. We solve the Generalized grandfather problem: for all $n \geq 0$, there exists a configuration that has an $n$th predecessor but not an $(n+1)$st one. We also solve (one interpretation of) the Unique father problem: there exists a finite stable configuration that contains a finite subpattern that has no predecessor patterns except itself. In particular this gives the first example of an unsynthesizable still life. The new key concept is that of a spatiotemporally periodic configuration (agar) that has a unique chain of preimages; we show that this property is semidecidable, and find examples of such agars using a SAT solver.

Our results about the topological dynamics of Game of Life are as follows: it never reaches its limit set; its dynamics on its limit set is chain-wandering, in particular it is not topologically transitive and does not have dense periodic points; and the spatial dynamics of its limit set is non-sofic, and does not admit a sublinear gluing radius in the cardinal directions (in particular it is not block-gluing). Our computability results are that Game of Life's reachability problem, as well as the language of its limit set, are PSPACE-hard.
\end{abstract}

\section{Introduction}

  Conway's Game of Life is a famous two-dimensional cellular automaton defined by John Horton Conway in 1970 and popularized by Martin Gardner~\cite{Ga70}. A cellular automaton can be thought of as zero-player game: the board is set up, and a simple rule determines the dynamics. In the case of Game of Life, the board is the two-dimensional infinite grid, where some grid cells are \emph{live}, and some are \emph{dead} (or \emph{empty}); the evolution rule, executed simultaneously in all cells, is that a dead cell becomes live if it has exactly three live (cardinal or diagonal) neighbors, and a live cell stays live if and only if it has two or three live neighbors.
  
  Iterating this rule gives rise to very complicated dynamics. Engineering patterns with interesting behaviors, and searching for such patterns by computer, has been an ongoing effort since the invention of the rule. For readers interested in delving into this world, we cite the very recent (and freely available) book \cite{JoGr22} of Johnston and Greene. One result that exemplifies the complexity of Game of Life is that it is \emph{intrinsically universal} \cite{DuRo99}, meaning that Game of Life can simulate any two-dimensional cellular automaton $f$ (including proper self-simulation), so that the states of $f$ correspond to large blocks with special content, and one step of $f$ is simulated in multiple steps of Game of Life.

  Game of Life can be thought of as a mathematical \emph{complex system}, namely it is a system where complex global behavior arises from interacting (simple) local rules. Such systems can be notoriously difficult to study. We can often use computer simulations to make empirical observations about typical and eventual behavior, but it can be very difficult to actually prove that a particular behavior persists on larger scales (even if it seems like its failure would require a massive conspiracy). Usually one can only successfully analyze systems that are very simple \cite{Fu17}, or their behavior simulates a phenomenon that is mathematically well-understood, say of an algebraic \cite{DeFoGrMa20} or number theoretic \cite{Ka12} nature, or the systems are specifically constructed for some purpose \cite{Lu10}. Due to intrinsic universality, it seems unlikely that Game of Life fits in any of these classes. 
  
  Indeed, for Game of Life, despite decades of study by enthusiasts, almost no non-trivial mathematical results exist that state limitations on its eventual behavior. In other words, as a dynamical system, we know very little about it. From computer simulations, one can conclude that Game of Life is highly ``chaotic'', and one can make educated guesses about things like the typical population density after a large number of iterations; however, it is very hard to make such claims rigorous. Rigorous results about Game of Life do exist, but they concern mostly the behavior of Game of Life on nice configurations (the engineering feats discussed above are of this type), and no known pattern behaves predictably in a general context; alternatively, they deal with one-step or static behavior \cite{El99}.

  In this paper, we study Game of Life through its \emph{agars}, which are the Game of Life community's term for spatiotemporally periodic points. More specifically, we observe that a simple algorithm (essentially Wang's partial algorithm from~\cite{Wa61}) can be used to find all agars with small enough periodicity parameters that have a unique chain of predecessors. We then study finite patches of these agars, and find some with interesting backwards forcing properties. Namely, these patterns behave deterministically in the (a priori nondeterministic) backwards dynamics of Game of Life. This intuitively allows us to study the ``last iterations'' of Game of Life (after an unknown number of steps), and leads to a wealth of results about how Game of Life behaves ``in the limit''. 
   
  One practical difficulty is that the algorithm we use is not of the usual kind, but rather it is in the class FP$^\text{NP}$ of problems that are solvable in polynomial time with an NP oracle. Throughout this work, modern SAT solvers have constantly impressed and even humbled us by how freely they can be used as such oracles.\footnote{For example, in our experience, general-purpose constraint-solvers and our own CA-specific solvers often fail even on the basic problem of finding a Game of Life preimage, while SAT solvers happily tell us, say, whether a preimage exists with particular constraints, and can find preimages for higher powers of Life.} While their role is not very explicit in the final write-up of the paper, this work would not have been possible without them.
  
  We do not expect the method of studying the eventual dynamics through self-enforcing patterns to be specific to Game of Life. Indeed, one can apply it directly to any cellular automaton rule (with any number of dimensions), and the idea can presumably be adapted to other systems as well. The reason we study a single example cellular automaton is that the results require us to find a ``witness'', usually a self-enforcing agar, and there is no \emph{a priori} bound on how long this pattern-crunching will take -- or whether it will succeed at all -- for a particular rule. A single agar also tends to only work for a single rule or trivial modifications thereof. The choice of precisely Game of Life as the example rule to study is not mathematically motivated, it is simply a well-known simple rule that has already been studied extensively. Some of our programs are available on GitHub at~\cite{Program}, for readers interested in trying the methods out on other rules.
    

  \subsection{The protagonists}
  \label{sec:Protagonists}
  
\begin{figure}[htp]
\begin{center}
  \pgfplotstableread{pattern_koynnos.cvs}{\matrixfile}
  \begin{subfigure}[b]{.27\linewidth}
  \begin{center}\begin{tikzpicture}[scale=0.18]
  \drawit{}
  \end{tikzpicture}\end{center}
  \vspace{-0.3cm}
  \caption{\Koynnos{}.}\label{fig:koynnos}
  \end{subfigure}
  \;\;
  \pgfplotstableread{pattern_kynnos.cvs}{\matrixfile}
  \begin{subfigure}[b]{.27\linewidth}
  \begin{center}\begin{tikzpicture}[scale=0.18]
  \drawit{}
  \end{tikzpicture}\end{center}
  \vspace{-0.3cm}
  \caption{\Kynnos{}.}\label{fig:kynnos}
  \end{subfigure}
\;\;
  \pgfplotstableread{pattern_marssiorkesteri.cvs}{\matrixfile}
  \begin{subfigure}[b]{.3\linewidth}
  \begin{center}\begin{tikzpicture}[scale=0.18]
  \drawit{}
  \end{tikzpicture}\end{center}
  \vspace{-0.3cm}
  \caption{Marching band.}\label{fig:marssiorkesteri}
  \end{subfigure}
\end{center}
  \caption{Patches of the agars. A $3$-by-$3$ grid of the repeating patterns is shown for each. Cyan cells are live.}
  \label{fig:Protagonists}
\end{figure}
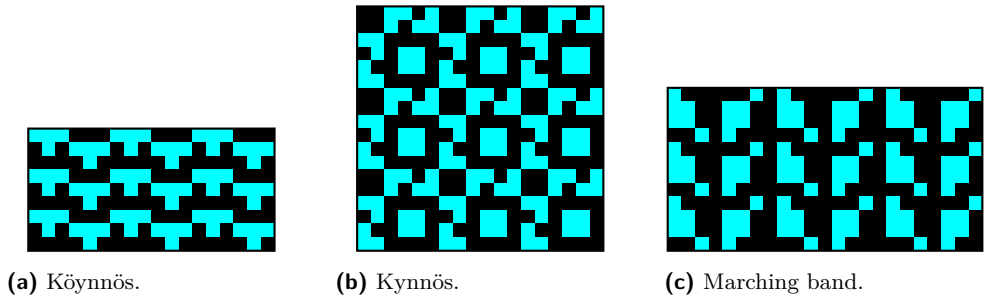

  We begin with a brief discussion of the agars that we use to prove our results. These will be explained in more detail in separate sections.
  
  Figure~\ref{fig:koynnos} shows the pattern we call \koynnos{}\footnote{Finnish for vine.}. The infinite agar obtained by repeating this pattern infinitely in each direction has no preimage other than itself; we say it is \emph{self-enforcing}. Up to symmetries there are exactly $11$ self-enforcing agars of size $6 \times 3$. \Koynnos{} has the special property that it is impossible to stabilize a finite difference to this configuration: if one modifies the agar in finitely many cells, the difference spreads at the speed of light (one column of cells per time step, which is the maximal speed at which information can be transmitted by Game of Life). We say it \emph{cannot be stabilized from the inside}.
  
  Figure~\ref{fig:kynnos} shows the pattern we call \kynnos{}\footnote{Finnish for that which is tilled.}. The corresponding infinite agar is again self-enforcing. Up to symmetries there are at least 52 self-enforcing agars of size $6 \times 6$ (we were unable to finish the search, so it is possible that more exist). \Kynnos{} has two special properties. First, it contains a finite patch such that if a configuration has this patch in its image, then the configuration already had that patch in place, i.e.\ one cannot synthesize it from any other patch. Second, unlike \koynnos{}, it can be stabilized from the inside.
  
  Figure~\ref{fig:marssiorkesteri} shows the pattern we call marching band\footnote{English for marssiorkesteri.}. This agar has temporal period two. Its most important property is that an infinite south half-plane of this pattern must shrink if there is a difference on its border, meaning that in the nondeterministic inverse dynamics of Game of Life, an infinite south half-plane of this pattern ``marches'' to the north.
  
  To find the marching band, we searched through all $w \times h$-patterns such that the corresponding agar with periods $(w,0)$ and $(0,h)$ is temporally (exactly) $t$-periodic, for the parameter range $2 \leq w \leq 9$, $2 \leq h \leq 5$, $2 \leq t \leq 3$. There were no self-enforcing agars with temporal period $3$ in this range, and there were exactly $14$ self-enforcing agars with period $2$. The marching band is the only one that has the marching property in any direction.

\subsection{Results}
\label{sec:Results}

Denote by $g : \{0,1\}^{\Z^2} \to \{0,1\}^{\Z^2}$ the Game of Life cellular automaton, where dead cells are represented by 0 and live cells by 1. In this section, we list all our new technical contributions about $g$. The reader should consult Section~\ref{sec:def} for precise definitions of terms used in this section. First, we solve the Generalized grandfather problem: for all $n \geq 0$, there exists a configuration that has an $n$th predecessor but not an $(n+1)$st one.
  
  \begin{theorem}[Generalized grandfather problem]
  \label{thm:Unstable}
  For each $n \geq 0$, there exists $x \in \{0,1\}^{\Z^2}$ with $g^{-n}(x) \neq \emptyset$ and $g^{-(n+1)}(x) = \emptyset$.
\end{theorem}

The case of $n = 0$ (that $g$ is not surjective) was resolved by R.\ Banks in 1971, only a year after the introduction of Game of Life. Conway stated the Grandfather problem, namely the case $n = 1$ of the above, in 1972, and promised \$50 in the \emph{Lifeline} newsletter \cite{Wa72} for its solution. This stayed open until 2016, when the cases $n \in \{1, 2, 3\}$ were proved by the user mtve of the ConwayLife forum. We note (see Lemma~\ref{lem:finite-unstable} for the proof) that while Theorem~\ref{thm:Unstable} refers to infinite configurations, the analogous statements for finite patterns or finite-population configurations are equivalent to it.
  Cellular automata satisfying the conclusion of Theorem~\ref{thm:Unstable} are sometimes called ``unstable'' \cite{Ma95}, though we avoid this terminology here, as ``stable'' has another meaning in Game of Life jargon.
  
  More specifically, we prove the following two results, which strengthen Theorem~\ref{thm:Unstable} in different directions. The first result is proved using \koynnos{} and is based on the fact it cannot be stabilized from the inside.
  The notation $g^{-n}(p)$ for a finite pattern $p$ of shape $D \subset \Z^2$ stands for the set of patterns of shape $D + [-n,n]^2$ that evolve into $p$ in $n$ steps.
  
\begin{restatable}{thm}{polytime}
  \label{thm:polytime}
  There exists a polynomial time algorithm that, given $n \geq 0$ in unary, produces a finite pattern $p$ with $g^{-n}(p) \neq \emptyset$ and $g^{-(n+1)}(p) = \emptyset$.
\end{restatable}

The algorithm is very simple: change the value of one cell in the agar, apply the Game of Life rule $n$ times, and pick the central $[-30-6n, 30+6n] \times [-27-8n, 27+8n]$-patch of the resulting configuration as $p$. An example with $n = 28$ is shown in Figure~\ref{fig:orphan28} (with insufficient padding: the periodic background should extend 164 cells further to the left and right, and 220 cells up and down).

\begin{figure}[htp]
\begin{center}
  \pgfplotstableread{orphan_28.cvs}{\matrixfile}
  \begin{tikzpicture}[scale=0.1]
  \drawit{}
  \end{tikzpicture}
  \vspace{-0.2cm}
  \caption{A ``level-$29$ orphan'' obtained by perturbing \koynnos{}: these angry deities could be found 28 seconds after the Big Bang, then went extinct.}\label{fig:orphan28}
  \end{center}
\end{figure}
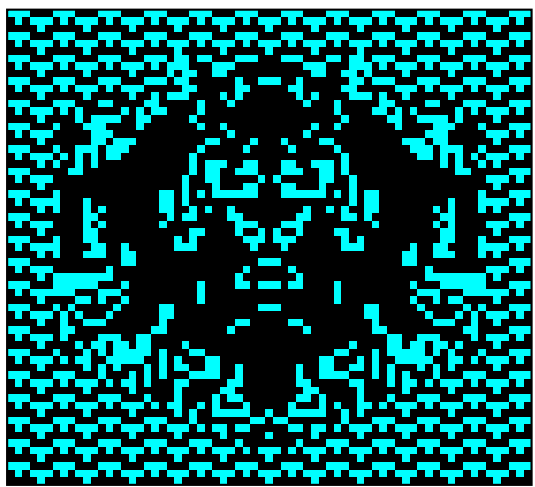

The second result is proved using \kynnos{}, and is based on the facts that \kynnos{} admits a self-enforcing patch and that it can be stabilized from the inside. The following was first pointed out by Adam Goucher \cite{Go22forum}.

  \begin{restatable}{thm}{rateofgrowth}
  \label{thm:rateofgrowth}
  For any large enough $n$, there exists an $n \times n$ pattern which appears in the $k$th image of Game of Life, but does not appear in its $(k+1)$st image, where $2^{n^2/368 - O(n)} \leq k \leq 2^{n^2}$.
  \end{restatable}

  This is (up to a suitable equivalence relation) the optimally slow growth rate for higher level orphans. The same idea can be used to obtain the following results about the limit set of Game of Life.
  Also called the eventual image, it is the set of configurations with arbitrarily long chains of predecessors.
  The language of the limit set refers to the set of finite patterns occurring in it.

  \begin{theorem}
  \label{thm:pspacehard}
  The limit set of Game of Life has PSPACE-hard language.
  \end{theorem}
  
  The language might well be much harder.
  Even for one-dimensional cellular automata it can be $\Pi^0_1$-complete \cite{CuPaYu89,Hu87}; we do not know if Game of Life reaches this upper bound.

  We also obtain information about the symbolic dynamical nature of the limit set.
  A set of configurations is sofic if it can be defined by Wang tiles, or squares with colored edges: in a valid tiling of $\Z^2$, colors of adjacent edges are required to match, and the tiles can additionally be marked with 0 and 1 to project each valid tiling to a binary configuration. The set of those projections is called a sofic shift. Sofic systems form a large and varied class of subshifts, for example their one-dimensional projections can be essentially arbitrary (subject only to an obvious computability condition) \cite{DuRoSh10,AuSa13}.
  We show that the limit set of Game of Life cannot be defined by a tile set in this way.
  
  \begin{theorem}
  \label{thm:nonsofic}
  The limit set of Game of Life is not sofic.
  \end{theorem} 
  
  Besides illuminating the iterated images of Game of Life and its limit set, the self-enforcing \kynnos{} patch itself solves a second open problem, namely the Unique father problem stated by John 
  Conway in \cite{Wa72,Co92}: is there a still life (a finite-population configuration that is a fixed point of $g$) whose only predecessor is itself, ``with some fading junk some distance away not being counted''?
  We solve one interpretation of this problem.
  
  \begin{theorem}[Unique father problem]
  \label{thm:SelfEnforcing}
  There exists a finite still life configuration $x$ that contains a finite subpattern $p$ such that every preimage of $x$ also has subpattern $p$.
  \end{theorem}

  One can also imagine stronger variants of the Unique father problem: for example, we could require $p$ to contain all live cells of $x$, or all cells in their convex hull. These stay open.

  Theorem~\ref{thm:SelfEnforcing} also tells us something about the dynamics of Game of Life \emph{restricted to} its limit set, i.e.\ its limit dynamics.
  The chain-wandering property essentially means that there is a finite pattern that occurs in the limit set of Game of Life, but never returns to itself under the dynamics no matter how we fill the surrounding infinite plane.
  In fact, we are even allowed to completely rewrite the entire configuration on every step, apart from the domain of the pattern.
 
  \begin{theorem}
    \label{thm:NotChainTrans}
 Game of Life is chain-wandering 
 on its limit set. 
 \end{theorem}
 
 Much is known about the kinds of things that can happen in Game of Life orbits, in particular it is well known that Game of Life is computationally universal and can simulate any cellular automaton. Nevertheless, to our knowledge all existing methods of simulating unbounded computation require the rest of the configuration to be empty (or at least stay out of the way).
 With our methods, we can enforce computations in a finite region (conditioned on its end state) even when it is completed into an infinite configuration by an adversary.
 
 \begin{theorem}
  \label{thm:pspacehardreach}
  The reachability problem of Game of Life is PSPACE-hard, i.e.\ given two finite patterns $p, q \in \{0,1\}^D$ whose domain $D$ has polynomial extent, it is PSPACE-hard to tell whether there exists a configuration $x$ with $x|D = p$ and $g^n(x)|D = q$ for some $n \geq 0$.
  \end{theorem}
  
  Finally, the properties of the marching band's backwards dynamics imply that the limit set contains patterns that cannot be ``glued'' together too close: there are two $n \times n$ patterns such that no configuration of the limit set contains both of them separated by a distance less than $n/15$.

 \begin{theorem}
 \label{thm:nogluing}
  For all large enough $n$ there exist patterns $p, q \in \{0,1\}^{[0,n-1]^2}$ such that $p$ and $q$ appear in the limit set, but $p \sqcup \sigma_{(0, \lfloor n/15 \rfloor)}(q)$ does not.
 \end{theorem}

 \begin{corollary}
 The limit set of Game of Life is not block-gluing (thus has none of the gluing properties listed in \cite{BoPaSc10}).
 \end{corollary}

\subsection{Programs}

Some of the programs we used can be found on GitHub at~\cite{Program}. We have included Python scripts enumerating self-enforcing agars, and scripts checking the claimed properties of our three agars. In particular one can find implementations of Algorithms~\ref{alg:Hat} and~\ref{alg:Patch}. The scripts use the PySAT~\cite{Pysat017dev15} library to call the Minisat~\cite{Minisat22} SAT solver (the library supports many other solvers as well).

\section{Definitions}
\label{sec:def}

Our intervals are discrete.
To simplify formulas, we denote by
\[
  \annu{a}{b}{c}{d}{e}{f}{g}{h} = ([-a,b]\times[-c,d]) \setminus ([-e,f]\times[-g,h])
\]
a rectangular discrete annulus when the second rectangle fits fully inside the first, that is, $-a \leq -e \leq f \leq b$ and $-c \leq -g \leq h \leq d$.

We assume some familiarity with topological and symbolic dynamics and give only brief definitions, see e.g.\ \cite{LiMa95} for a basic reference. We denote by $S$ a finite \emph{alphabet}. A \emph{configuration} or \emph{point} is an element of $S^{\Z^d}$. More generally, a \emph{pattern} (or sometimes \emph{patch} in more informal contexts) is a function $p : \dom(p) \to S$, where $\dom(p) \subset \Z^d$ is the \emph{domain} of $p$. If $S \subset \N$, then by $\sum p$ we denote the sum $\sum_{\vec v \in \dom(p)} p(\vec v)$. For $\vec v \in \Z^d$, a pattern $p$ and $D \subset \Z^d$, 
we write $q = p|D$ for the restriction $\dom(q) = D \cap \dom(p), q(\vec v) = p(\vec v)$. A pattern is \emph{finite} if its domain is, and a configuration is \emph{finite} if its sum as a pattern is finite. If $p, q$ are patterns with disjoint domains, define $p \sqcup q = r$ by $\dom(r) = \dom(p) \cup \dom(q), r|\dom(p) = p, r|\dom(q) = q$. The \emph{extent} of a pattern is the minimal hypercube containing the origin and its domain. For two patterns, write $\mathrm{eq}(q, q')$ for the set of vectors $\vec v \in \dom(q) \cap \dom(q')$ such that $q(\vec v) = q'(\vec v)$, and $\diff(q,q')$ for those that satisfy $q(\vec v) \neq q'(\vec v)$. For computer science purposes, we note that patterns with polynomial extent have an efficient encoding as binary strings.

The \emph{full shift} is the set of all configurations $S^{\Z^d}$ with the product topology (where $S$ has the discrete topology), under the action of $\Z^d$ by homeomorphisms $\sigma_{\vec v}(x)_{\vec u} = x_{\vec v + \vec u}$ called \emph{shifts}. We use the same formula to define $\sigma_{\vec v}(p)$ for patterns $p$ (of course shifting the domain correspondingly). A pattern $p$ defines a \emph{cylinder} $[p] = \{x \in S^{\Z^d} \;|\; x|\dom(p) = p\}$. Cylinders defined by finite patterns are a base of the topology, and their finite unions are exactly the clopen sets. The \emph{symbol partition} is the clopen partition $\{[s] \;|\; s \in S\}$ where $s$ is identified with the pattern $p : \{\vec 0\} \to S$ with $p(\vec 0) = s$. The space $S^{\Z^d}$ is homeomorphic to the Cantor space, and is metrizable. One possible metric is $\dist : (S^{\Z^2})^2 \to \mathbb{R}$, $\dist(x, y) = 2^{-\sup\{n \;|\; x|[-n,n]\times[-n,n] = y|[-n,n]\times[-n,n]\}}$ with $2^{-\infty} = 0$.

A \emph{cellular automaton} (or \emph{CA}) is a continuous self-map $f : S^{\Z^d} \to S^{\Z^d}$ that commutes with the shifts. The \emph{neighborhood} is a set $N \subset \Z^d$ such that $f(x)_{\vec 0}$ is determined by $x|N$; a finite neighborhood always exists by the Curtis-Hedlund-Lyndon theorem~\cite{He69}. It is easy to show that there is always a unique \emph{minimal neighborhood} under inclusion. A state $0 \in S$ is \emph{quiescent} if $f(0^{\Z^d}) = 0^{\Z^d}$. A \emph{subshift} is a closed subset $X$ of $S^{\Z^d}$ invariant under shifts. Its \emph{language} is the set of finite patterns $p$ such that $[p] \cap X \neq \emptyset$, and we say these patterns \emph{appear} or \emph{occur} in the subshift. Patterns that do not appear in $f(S^{\Z^d})$ are usually called \emph{orphans}. We say $p$ is a \emph{level-$n$ orphan} if it appears in $f^{n-1}(S^{\Z^d})$ but not in $f^n(S^{\Z^d})$ (so the usual orphans are level-$1$). The \emph{limit set} of a cellular automaton $f$ is $\Omega(f) = \bigcap_n f^n(S^{\Z^d})$. It is a subshift invariant under $f$. A \emph{subshift of finite type} is a subshift of the form $\bigcap \sigma_{\vec v}(C)$ where $C$ is clopen. A \emph{sofic shift} is a subshift which is the image of a subshift of finite type under a shift-commuting continuous function.

We are mainly interested in $d = 2$, $S = \{0,1\}$, and the \emph{Game of Life} cellular automaton $g : \{0,1\}^{\Z^2} \to \{0,1\}^{\Z^2}$ defined by
\begin{align*}
g(x)_{\vec v} = 1 \iff (&x_{\vec v} = 0 \wedge \sum(x|\vec v+K) = 3) \\
\vee (&x_{\vec v} = 1 \wedge \sum(x|\vec v+K) \in \{2,3\}),
\end{align*}
where $K = [-1,1]^2 \setminus \{(0,0)\}$.

A \emph{fixed point} (of a CA $f$) is $x \in S^{\Z^d}$ such that $f(x) = x$. In the context of Game of Life these are also called \emph{stable configurations} or \emph{still lifes}. \emph{Spatial} and \emph{temporal} generally refer respectively to the $\Z^d$-action of shifts and the action of a CA. In particular a \emph{spatially periodic point} is a configuration $x \in S^{\Z^d}$ which has a finite orbit under the shift dynamics, and \emph{temporal periodicity} means $f^n(x) = x$ for some $n \geq 1$. Spatiotemporal periodicity means that both hold; in the Game of Life context spatiotemporal points are also called \emph{agars}.

If $f : X \to X$ is a continuous function, an \emph{$\epsilon$-chain from $x$ to $y$} is $x = x_0, x_1, \ldots, x_k = y$ with $k \geq 1$ such that $\dist(f(x_i), x_{i+1}) < \epsilon$ for $0 \leq i < k$. We say $f$ is \emph{chain-nonwandering} if for all $\epsilon > 0$ and $x \in X$ there is an $\epsilon$-chain from $x$ to itself; otherwise $f$ is \emph{chain-wandering}. (In the literature, chain-nonwandering is more commonly known as chain-recurrence, but both terms are logical.) We say $f$ is \emph{topologically transitive} if for all nonempty open sets $U,V$ we have $f^n(U) \cap V \neq \emptyset$ for some $n$. It is \emph{sensitive} (to initial conditions) if there exists $\epsilon > 0$ such that for all $x \in X$ and $\delta > 0$ there exists $y \in X$ with $\dist(x,y) < \delta$ and $n \in \N$ such that $\dist(f^n(x), f^n(y)) \geq \epsilon$. We say $f$ has \emph{dense periodic points} if its set of temporally periodic points is dense.

\section{Proofs}

We begin by introducing a formalism for forced cells in the preimages of a given pattern or configuration.
The general topological idea is the following: if we have a zero-dimensional space $X$ and a family of closed sets $\II$ which is closed under arbitrary intersections and contains the empty set, then to any continuous $f : X \to X$ we can associate a map $\hat{f} : \II \to \II$ by
\begin{equation}
  \label{eq:generalhat}
  \hat f(A) = \bigcap \{B \in \II \;|\; f^{-1}(A) \subset B\}.
\end{equation}
We call this the \emph{dual map} of $f$ with respect to $\II$.

In our situation, $X = S^{\Z^d}$ and $f : S^{\Z^d} \to S^{\Z^d}$ is a cellular automaton. We define three families of subsets of $S^{\Z^d}$:
\begin{itemize}
\item $\II$ consists of the cylinders $[p] \subset S^{\Z^d}$ defined by all patterns $p$, plus the empty set $\emptyset$, which we denote by $\top$. The entire space $S^{\Z^d}$, which is the cylinder of the empty pattern, is denoted by $\bot$.
\item $\FF \subset \II$ consists of all cylinders $[p]$ defined by finite patterns $p$.
\item $\CC \subset \II$ consists of the singletons $[x] = \{x\}$ for full configurations $x \in S^{\Z^d}$.
\end{itemize}
Note that $\FF \cap \CC = \emptyset$. The family $\FF$ is naturally stratified into finite subsets $\FF_M = \{ [p] \mid p \in S^M \}$, where $M \subset \Z^d$ ranges over finite sets.
For a cylinder $[p] \in \II$, equation~\eqref{eq:generalhat} defines $\hat f([p]) \in \II$ as the cylinder $[q]$, where $q$ contains exactly those cells whose values are the same in all $f$-preimages of $p$, or $\top$ if $p$ has no $f$-preimages (i.e.\ is an orphan).
Also, $\hat f(\top) = \top$.

\begin{example}
  Consider $S = \{0,1,2\}$ and the cellular automaton $f : S^\Z \to S^\Z$ defined by
  \[
    f(x)_0 =
    \begin{cases}
      2, & \text{if~} x_0 = 2, \\
      \min(x_0, x_1), & \text{otherwise.}
    \end{cases}
  \]
  The minimal neighborhood of $f$ is $N = \{0,1\}$.
  The pattern $p = 0 0 2$ of domain $\{0,1,2\}$ has preimages $f^{-1}(p) = \{ 0 0 2 0, 0 0 2 1, 0 0 2 2, 1 0 2 0, 1 0 2 1, 1 0 2 2 \}$ of domain $\{0,1,2,3\}$.
  Thus $\hat f([p]) = [q]$, where $q = 0 2$ has domain $\{1,2\}$, since the values of these cells are the same in all preimages.
  The pattern $p' = 1 0 2$ has no preimages, so $\hat f([p']) = \top$.
\end{example}

We define a partial order on $\II$ by $[p] \leq [q]$ whenever $[q] \subset [p]$, and $\alpha \leq \top$ for all $\alpha \in \II$.
The intuition is that $[p] \leq [q]$ corresponds to the pattern $q$ specifying more cells than $p$, and thus containing more information.
As the empty set $\top$ in a sense specifies the maximal amount of information -- a contradiction -- it is the largest element.
Note that $\CC$ consists of the maximal elements of $\II \setminus \{\top\}$.

We give $\II$ the topology with basis sets $U_p = \{\alpha \in \II \;|\; [p] \leq \alpha\}$ for $[p] \in \FF$ as well as $\{\top\}$, making $\top$ an isolated point. This space is not Hausdorff ($T_2$), indeed it only satisfies the Kolmogorov ($T_0$) separation axiom. The induced topology on $\CC$ is the standard compact Cantor topology, and $\FF$ is a dense subset of $\II \setminus \{\top\}$. Every nonempty open set contains $\top$: ``the contradiction is dense''.

\begin{lemma}
  \label{lem:continuous}
  The dual map $\hat f : \II \to \II$ is continuous.
\end{lemma}

We are simply saying that if a (possibly infinite) pattern forces some particular value in some cell in the preimage, then actually some finite patch already forces it. 
The proof is a straightforward compactness argument.

\begin{proof}
  Continuity at $\top$ is obvious.
  We show continuity at a cylinder $[p] \in \II$.
  Suppose first that $\hat f([p]) = [q]$, and let $[r] \in \FF$ be such that $[q] \in U_r$.
  This means that $p$ forces the pattern $q$ in its $f$-preimages, and $r$ is a finite subpattern of $q$.
  There exists a finite subpattern $s$ of $p$ that forces $r$, for otherwise we could take larger and larger subpatterns of $p$ along with two preimages that disagree on $\dom(r)$, and in the limit obtain two preimages of $p$ that disagree on $\dom(r)$.
  Hence $\hat f(U_s) \subset U_r$.

  Suppose then that $\hat f([p]) = \top$, meaning that $p$ is an orphan.
  It is well known that $p$ contains a finite subpattern $r$ that is also an orphan.
  Then $\hat f(U_r) = \{\top\}$.
\end{proof}

We list some other easy properties of $\hat f$. For $\alpha, \beta \in \II$ write $\alpha \parallel \beta$ for $\alpha \cap \beta \neq \top$. In the case of cylinders, this means that the corresponding patterns agree on the intersection of their domains. For a pattern $p$, write $f(p)$ for the pattern obtained by applying the local rule of $f$ (with the minimal neighborhood) in every position whose neighborhood is contained in the domain of $p$ (and only those positions are included in the domain of $f(p)$). 
Note that with this definition $f([p]) \subset [f(p)]$, and the inclusion may be strict.

\begin{lemma}
\label{lem:HatBasics}
\begin{itemize}
\item $\FF \cup \{\top\}$ is preserved under $\hat f$. Indeed, we have $\hat f(\FF_M) \subset \FF_{M + N} \cup \{\top\}$ where $N$ is the minimal neighborhood of $f$.
\item $\hat f$ is monotone, i.e.\ $\alpha \leq \beta$ implies $\hat{f}(\alpha) \leq \hat{f}(\beta)$.
\item $\hat f \circ \hat g \leq \widehat{f \circ g}$ pointwise.
\item for all $\alpha \in \II$, either $\hat f(\alpha) = \top$ or $\hat f(\alpha) = [p]$ with $[f(p)] \leq \alpha$; in particular $[f(p)] \parallel \alpha$ in the latter case.
\item for all $\alpha \in \II$, we have $\hat f(\sigma_{\vec v}(\alpha)) = \sigma_{\vec v}(\hat f(\alpha))$.
\end{itemize}
\end{lemma}

The next few results refer to $\FP^{\NP}$, the class of function problems solvable in deterministic polynomial time with the help of an oracle that can solve an $\NP$ decision problem in one step.
Of course, the oracle can be invoked repeatedly to construct $\NP$ certificates in a polynomial number of steps.
This class naturally captures the method of using SAT solvers as black boxes to compute preimages of finite patterns.

\begin{lemma}
  \label{lem:ComputableHat}
  For a fixed CA $f$, given $p \in \FF$, the image $\hat f([p])$ can be computed in $\FP^{\NP}$.
  It remains computable if $f$ is also given as input.
\end{lemma}

\begin{proof}
  Since $\hat f(\FF_M) \subset \FF_{M + N}$, we only need to determine which coordinates in $M + N$ are forced in preimages.
  This requires at most $1+|M + N|$ calls to an NP oracle: one to request a preimage, and for each $\vec v \in M+N$, one to request a pair of preimages which differ at $\vec v$.
\end{proof}

The proof above is the easiest way to get the theoretical result, but for practical purposes we give Algorithm~\ref{alg:Hat}, which tends to find the $\hat f$-image much quicker (and is just as quick to implement). It is written for an ``incremental oracle'', meaning we can only \emph{add} constraints to it (represented by the set $F$) when we make a new query. In this case, we compute a single $f$-preimage $q$ of the input pattern $p$, and then compute additional preimages that differ from $q$ on progressively smaller sets of cells. Modern SAT solvers tend to support such incremental access -- of course, on the side of theory it is easy to see that the class $\FP^{\NP}$ is the same whether or not queries are restricted to be incremental.

\begin{algorithm}[htp]
  \caption{Finding $\hat f([p])$ for a finite pattern $p \in S^M$.}
\label{alg:Hat}
\begin{algorithmic}
  \Function{HatCA}{$f, p$}
    \State Let $\mathcal{O} \gets \text{NP oracle}$.
    \If{$\mathcal{O}$ finds a pattern $q \in f^{-1}(p)$}
      \State Let $D \gets M + N$.
      \State Let $F \gets \{(q, D)\}$.
      \Loop
        \If{$\mathcal{O}$ finds a pattern $q' \in f^{-1}(p)$ with $q'|E \neq r|E$ for all $(r,E) \in F$}
          \State Let $D \gets D \cap \mathrm{eq}(q, q')$.
          \State Let $F \gets F \cup \{(q',D)\}$
        \Else
          \State \textbf{return} $q|D$
        \EndIf
      \EndLoop
    \Else
      \State \textbf{return} $\top$
    \EndIf
  \EndFunction
\end{algorithmic}
\end{algorithm}

Our results rely on the existence of patterns $p$ that force large patterns into their preimages, meaning that $\hat f([p])$ is large in the sense of $\leq$.
We say a pattern $p$ is \emph{self-enforcing} under $f$ if $[p] \leq \hat f([p])$.
In a slight abuse of terminology, we also say that a temporally $t$-periodic configuration $x \in S^{\Z^d}$ is \emph{self-enforcing} if $\widehat{(f^t)}([x]) = [x]$.
A self-enforcing agar is then a spatially and temporally periodic configuration that has a unique chain of preimages.

\begin{lemma}
  \label{lem:SEARE}
  The set of all pairs $(f, x)$ such that $f$ is a CA on $S^{\Z^d}$ and $x \in S^{\Z^d}$ is a self-enforcing agar is recursively enumerable.
\end{lemma}

\begin{proof}
  Let $x \in S^{\Z^d}$ be a self-enforcing agar with spatial periods $n_1 \vec e_1, \ldots, n_d \vec e_d$ and temporal period $t$.
  Denote the iterated CA by $h = f^t$, and let $B = [0, n_1-1] \times \cdots \times [0, n_d-1]$.
  We need to find a certificate for $\hat{h}([x]) = [x]$.
  For this, observe that by continuity of $\hat h$ there is a finite subpattern $p$ of $x$ such that $\hat h([y]) \geq [x|B]$ for every configuration $y \in [p]$. This implies $\hat h([p]) \geq [x|B]$. By Lemma~\ref{lem:ComputableHat}, this latter inequality can be checked in $\FP^{\NP}$.

We claim that $p$ is a certificate that $x$ is a self-enforcing agar.
Let $\vec v \in V = \langle n_1 \vec e_1, \ldots, n_d \vec e_d \rangle$ be arbitrary.
We compute
\begin{align*}
  \hat h([x]) = \sigma_{\vec v}(\hat h([x])) \geq \sigma_{\vec v}(\hat h([p])) \geq \sigma_{\vec v}([x|B]) = [x|\vec v + B],
\end{align*}
and since $\Z^d = \bigcup_{\vec v \in V} (\vec v + B)$, this implies $\hat h([x]) = [x]$.
\end{proof}

\begin{remark}
\label{rem:SEARE}
The semi-algorithm described in the proof is not very practical: given an agar, we have no information about how large the certificate could be, so for each agar we either need to guess some certificate size, or we have to keep trying increasingly large certificates. Our implementation runs in parallel a search for other periodic preimages for the agar -- if such a preimage exists, then clearly the agar does not enforce itself, and we can stop looking for a certificate. We omit the pseudocode.

Most agars in the range we searched were either self-enforcing or had another periodic preimage. There exist two-dimensional cellular automata whose set of self-enforcing agars is not computable (by a relatively simple reduction from the tiling problem of Wang tiles~\cite{Be66}, which we omit), but we do not know whether this is the case for Game of Life.
\end{remark}

Say a pattern $p \subset S^M$ is \emph{locally fixed} for the CA $f$ if there exists a pattern $q \in S^{M+N}$ (where $N$ is the minimal neighborhood of $f$) such that $p = q|M = f(q)|M$.

\begin{lemma}
  For every CA $f$ on $S^{\Z^d}$, every locally fixed pattern $p \in S^M$ admits a unique maximal self-enforcing subpattern.
  For a fixed CA $g$, given $p$, a vector $\vec v \in \Z^d$ and $n \geq 1$ in unary, it can be computed in $\FP^\NP$ for the CA $f = \sigma_{\vec v} \circ g^n$.
\end{lemma}

\begin{proof}
  Since $p$ has finitely many subpatterns and the empty pattern is trivially self-enforcing, $p$ admits at least one maximal self-enforcing subpattern.
  If $D, D' \subset M$ satisfy $\hat{f}([p|D]) \geq [p|D]$ and $\hat{f}([p|D']) \geq [p|D']$, then $\hat{f}([p|D \cup D']) \geq [p|D \cup D']$ by monotonicity of $\hat{f}$.
  Thus $q = p|E$ for $E = \bigcup \{ D \subset M \mid \hat{f}([p|D]) \geq [p|D] \}$ is the unique self-enforcing subpattern.

  Then fix $g$, and let $p$, $\vec v$ and $n$ be given.
  We apply Algorithm~\ref{alg:Patch} to the CA $f = \sigma_{\vec v} \circ g^n$.
  On each iteration of the loop, the algorithm replaces $p$ with the maximal subpattern forced by $p$ (here we use the fact that $p$ is locally fixed).
  Since $q$ is a subpattern forced by itself, by monotonicity it is also forced by each of these subpatterns, and thus remains a subpattern on each iteration.
  Since $q$ is maximal and $p$ has finitely many subpatterns, the algorithm eventually converges on $q$.

  Finally, Algorithm~\ref{alg:Patch} is in $\FP^\NP$, since the number of iterations of the loop is at most $|M|$, and \textsc{HatCA}($f, p$) is in $\FP^\NP$ with respect to these parameters.
\end{proof}

\begin{algorithm}[htp]
\caption{Finding the maximal self-enforcing subpattern of a locally fixed pattern $p \in S^M$.}
\label{alg:Patch}
\begin{algorithmic}
  \Function{SelfEnforcingSubpattern}{$p$}
    \Loop
      \State Let $q \gets$ \Call{HatCA}{$f, p$}${}|M$.
      \If{$q = p$}
        \State \textbf{return} $q$
      \Else
        \State Let $p \gets q$.
      \EndIf
    \EndLoop
  \EndFunction
\end{algorithmic}
\end{algorithm}

To conclude this section, we show that in the formulation of the Generalized grandfather problem (which we prove as Theorem~\ref{thm:Unstable}), it makes no difference whether we consider unrestricted configurations, finite configurations or finite patterns.
This is well-known in cellular automata theory.

\begin{lemma}
  \label{lem:finite-unstable}
  Let $f : S^{\Z^d} \to S^{\Z^d}$ be a cellular automaton with a quiescent state $0 \in S$, and $n \in \N$.
  The following are equivalent:
  \begin{enumerate}
  \item
    There exists a finite configuration $x \in S^{\Z^d}$ such that $f^{-n}(x)$ contains a finite configuration, and $f^{-(n+1)}(x) = \emptyset$.
  \item
    There exists $x \in S^{\Z^d}$ such that $f^{-n}(x) \neq \emptyset$ and $f^{-(n+1)}(x) = \emptyset$.
  \item
    There exists a finite pattern $p$ such that $f^{-n}(p) \neq \emptyset$ and $f^{-(n+1)}(p) = \emptyset$.
  \end{enumerate}
\end{lemma}

\begin{proof}
  The implication 1 $\implies$ 2 is clear, and 2 $\implies$ 3 is the classical compactness argument that we used to prove Lemma~\ref{lem:continuous}.

  We prove 3 $\implies$ 1.
  Take an arbitrary $q \in f^{-n}(p)$, and complete it into a finite configuration $y \in S^{\Z^d}$ by setting $y_{\vec v} = 0$ for all $\vec v \in \Z^d \setminus \dom(q)$.
  Then $x = f^n(y)$ satisfies the conditions of item 1: $f^{-n}(x)$ contains the finite configuration $y$, while $f^{-(n+1)}(x) = \emptyset$ since $x$ contains an occurrence of $p$.
\end{proof}

\subsection{K\"oynn\"os}
\label{sec:koynnos}

We begin by studying \koynnos{}, which we recall is obtained from the $6 \times 3$ pattern
\[
  P =
  \begin{array}{cccccc}
    1 & 1 & 1 & 0 & 0 & 0 \\
    0 & 1 & 0 & 1 & 1 & 1 \\
    0 & 0 & 0 & 0 & 1 & 0
  \end{array}
\]
by repeating $P$ horizontally and vertically to define an infinite $6 \times 3$-periodic configuration $x^P \in \{0,1\}^{\Z^2}$.
Observe that every $0$ in $x^P$ is surrounded by exactly four $1$s, and every $1$ by exactly three $1$s.
Thus we have $g(x^P) = x^P$, so that $x^P$ is indeed an agar.
Moreover, we claim that $x^P$ has no other predecessors than itself: $g^{-1}(x^P) = \{x^P\}$. This is due to the following lemma.

\begin{lemma}
  \label{lem:KoynnosForcing}
  Let $x$ be in the spatial orbit of \koynnos{}. Then $\hat g(x|[-12,17] \times [-12,14]) \geq x|[-8, 13] \times [-9, 10]$. 
\end{lemma}

\begin{proof}
  Applying Algorithm~\ref{alg:Hat} to $\sigma_{\vec v}(x^P)|[-12,17] \times [-12,14]$ for all $\vec v \in [0,5] \times [0,2]$ gives the result. The intersection of the domains of the patterns $\hat g(x|[-12,17] \times [-12,14])$ for such $x = \sigma_{\vec v}(x^P)$ is shown in Figure~\ref{fig:DomainIntersection}, and clearly contains the rectangle $[-8, 13] \times [-9, 10]$. 
\end{proof}

Put concretely, the lemma states that if $R$ is a periodic continuation of $P$ of size $30 \times 27$ and $Q$ is its predecessor, then $P$ must occur at the center of $Q$ (and indeed many more cells are forced, even beyond what we state in the lemma). This is indeed a certificate for $x^P$ being a self-enforcing agar, as in the proof of Lemma~\ref{lem:SEARE}: for any predecessor $y \in g^{-1}(x^P)$ and cell $\vec v \in \Z^d$, Lemma~\ref{lem:KoynnosForcing} gives $\sigma_{\vec v}(y)|[-8, 13] \times [-9, 10] = \sigma_{\vec v}(x^P)|[-8, 13] \times [-9, 10]$, so in particular $y_{\vec v} = x^P_{\vec v}$.

\begin{figure}[htp]
\begin{center}
  \pgfplotstableread{domain_intersection.cvs}{\matrixfile}
  \begin{tikzpicture}[scale=0.2]
  \drawitt{}
  \draw[fill=blue,opacity=0.2] (5, -24) rectangle (27, -4);
  \end{tikzpicture}
  \vspace{-0.2cm}
  \caption{The intersection of the domains of $\hat g(x|[-12,17] \times [-12,14])$ for $x$ in the spatial orbit of \koynnos{}, drawn in white inside $[-13, 18] \times [-13, 15]$. The area $[-8, 13] \times [-9, 10]$ is highlighted in blue.} 
  \label{fig:DomainIntersection}
  \end{center}
\end{figure}
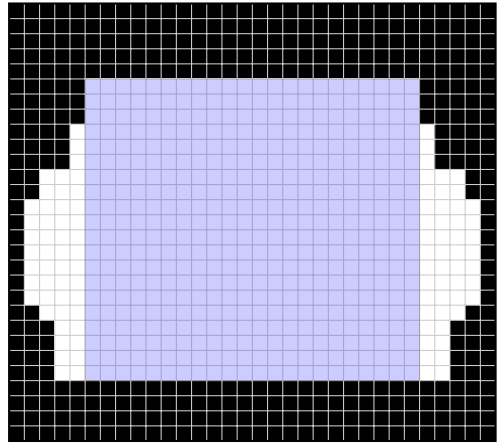


As a corollary of Lemma~\ref{lem:KoynnosForcing}, finite perturbations of $x^P$ can never be erased by $g$.
We prove a stronger claim: all finite perturbations spread to the left and right at a speed of one column per time step.
In particular, \koynnos{} cannot be stabilized from the inside.
We note that, as the agar \kynnos{} studied in the next section does not possess this property, we cannot use it to prove Theorem~\ref{thm:polytime}, at least with the same method.

\begin{lemma}
  \label{lem:explodev2}
  Consider a rectangle $R = [-n_W,n_E]\times[-n_S,n_N]$ and the surrounding annulus $A = \sannu{n_W+1}{n_E+1}{n_S+1}{n_N+1}{n_W}{n_E}{n_S}{n_N}$ of thickness 1.
  Let $p$ be a pattern such that the domain of $g(p)$ contains $A \cup R$, and suppose $p|A = g(p)|A = x^P|A$.
  If $\diff(x^P,g(p)) \cap R \subset [a, b] \times \Z$, then $\diff(x^P,p) \cap R \subset [a+1,b-1] \times \Z$.
\end{lemma}

Note that we may have $b - a \leq 1$, in which case the conclusion becomes $\diff(x^P,p) \cap R = \emptyset$, or equivalently, $x^P|R = p|R$.

\begin{proof}
  Since the orbit of \koynnos{} and $g$ are left-right symmetric, it is enough to prove that $\diff(x^P, p) \cap R \subset (-\infty,b-1]$.
  We prove the contrapositive: suppose there exists $(i, j) \in \diff(x^P,p) \cap R$ for some $i \geq b$, and let $i$ be maximal.
  We split into cases based on the congruence class of $i$ modulo 6, that is, the column of $P$ that $i$ lies in.
  Note that the bottom left cell of $P$ is at the origin in $x^P$, and the domain of $P$ is the rectangle $[0,5] \times [0,2]$.
  If $i \in \{0, 2, 3\} + 6 \Z$, we choose $j$ as maximal, and otherwise we choose it as minimal.

  We handle the case $i \in 2 + 6 \Z$, the others being similar or easier.
  If $j \in 3 \Z$, then $p_{(i+1,j+1)} = x^P_{(i+1,j+1)} = 1$ has four other $1$s in its neighborhood, and becomes $0$ in $g(p)$.
  If $j \in 1 + 3\Z$, then $p_{(i+1,j)} = x^P_{(i+1,j+1)} = 1$ has four or five other $1$s in its neighborhood, and becomes $0$ in $g(p)$.
  If $j \in 2 + 3 \Z$, then $p_{(i+1,j+1)} = x^P_{(i+1,j+1)} = 0$ has three $1$s in its neighborhood, and becomes $1$ in $g(p)$.
  In each case $\diff(x^P, g(p))$ intersects $\{i+1\} \times \Z$.
\end{proof}

For any $D \subset \Z^2$, the subpattern of \koynnos{} of shape $D + [0,29] \times [0,26]$ forces the subpattern of shape $D + [4,25] \times [3,22]$ to occur in its $g$-preimage, by Lemma~\ref{lem:KoynnosForcing}. By Lemma~\ref{lem:explodev2}, we force more: a non-\koynnos{} area inside a hollow patch of \koynnos{} expands horizontally under $g$, so under $\hat g$ the horizontal extent of the hole must shrink. We do not give a precise statement for this general fact, and only apply the lemma in the case of annuli. 

\begin{lemma}
  \label{lem:KoynnosAnnuli}
  Let $x$ be in the orbit of \koynnos{}. Suppose the following inequalities hold:
  \[ m_W - n_W \geq 30, m_E - n_E \geq 30, m_S - n_S \geq 27, m_N - n_N \geq 27. \]
  Denote $Q = x|\sannu{m_W}{m_E}{m_S}{m_N}{n_W}{n_E}{n_S}{n_N}$.
  If $n_E + n_W \geq 2$, then
  \begin{align*}
    \hat g(Q) \geq x|\annu{m_W-4}{m_E-4}{m_S-3}{m_N-4}{n_W-1}{n_E-1}{n_S+4}{n_N+3}
  \end{align*}
  while if $n_E+n_W \in \{0,1\}$ we have
  \[ \hat g(Q) \geq x|[-(m_W-4),m_E-4]\times[-(m_S-3),m_N-4] \]
\end{lemma}

One may consider the latter case a special case of the former: there too, the hole shrinks horizontally by two steps, and since its width is at most two it disappears.

\begin{proof}
  The inequalities simply state that the annulus $Q$ is thick enough that each of its cells is part of a $30 \times 27$ rectangle contained in $Q$. From Lemma~\ref{lem:KoynnosForcing}, we get $\hat g(Q) \geq x|A$, where $A = \sannu{m_W-4}{m_E-4}{m_S-3}{m_N-4}{n_W+4}{n_E+4}{n_S+4}{n_N+3}$ is a slightly thinner annulus.
  If there is no $g$-preimage for $Q$, then $\hat g(Q) = \top$ and we are done.
  Suppose then that it has a preimage $R$.
  Since $R \geq \hat g(Q) \geq x|A$, both $Q$ and $R$ agree with $x$ on the thickness-$1$ annulus $\sannu{n_W+5}{n_E+5}{n_S+5}{n_N+4}{n_W+4}{n_E+4}{n_S+4}{n_N+3} \subset A$.
  Lemma~\ref{lem:explodev2} implies that $R$ agrees with $x$ on $\sannu{m_W-4}{m_E-4}{m_S-3}{m_N-4}{n_W-1}{n_E-1}{n_S+4}{n_N+3}$, as claimed.
\end{proof}

We now prove Theorem~\ref{thm:polytime}, and thus give the first proof of Theorem~\ref{thm:Unstable}. In fact, we give a simple formula that produces configurations that have an $n$th preimage, but no $(n+1)$st one.

\begin{lemma}
Let $x$ be in the orbit of \koynnos{}, and suppose $\emptyset \neq \diff(y,x) \subset B = [0,a] \times [0,n]$ where $a \in \{0,1\}$. Then
\[ p = g^k(y)|[-30-6k, 30+a+6k] \times [-27-8k, 27+n+8k] \] 
appears in the $k$th image of $g$, but not in the $(k+1)$st.
\end{lemma}

\begin{proof}
By definition, $p$ appears in the $k$th image of $g$. It suffices to show its $\hat g^{k+1}$-image is $\top$. Namely, we then have $\widehat{g^{k+1}}(p) \geq \hat g^{k+1}(p) = \top$ by Lemma~\ref{lem:HatBasics}, which means precisely that $p$ has no $g^{k+1}$-preimage.

Let $q$ be the restriction of $p$ to
\begin{align*}
\annu{30+6k}{30+a+6k}{27+8k}{27+n+8k}{k}{a+k}{k}{n+k}.
\end{align*}
Observe that $q$ agrees with $x$ because $g$ has radius $1$, so by Lemma~\ref{lem:KoynnosAnnuli} and induction, we can deduce that
\[ \hat g^j(q) \geq x|\annu{30+6k-4j}{30+a+6k-4j}{27+8k-3j}{27+n+8k-4j}{k-j}{a+k-j}{k+4j}{n+k+3j} \]
for all $j \leq k$. This is because for $j \leq k-1$ we have
\begin{align*}
  30+6k-4j - (k-j) \geq {} & 30, & 30+a+6k-4j - (a+k-j) \geq {} & 30, \\
  27+8k-3j - (k+4j) \geq {} & 27, & 27+n+8k-4j - (n+k+3j) \geq {} & 27,
\end{align*}
and thus we can inductively apply the lemma. But in
\[ \hat g^k(q) \geq x|\annu{30+2k}{30+a+2k}{27+5k}{27+n+4k}{0}{a}{5k}{n+4k} \]
the annulus still has sufficient thickness (i.e.\ the inequalities still hold for $j = k$), so we can apply the second case of the lemma to get
\[ \hat g^{k+1}(q) \geq x|[-(26+2k),26+2k]\times[-(24+5k),23+n+4k] = r. \]

By Lemma~\ref{lem:HatBasics} we have $\hat g^{k+1}(p) \geq \hat g^{k+1}(q) \geq r$, and by the same lemma we either have $\hat g^{k+1}(p) = \top$ (as desired), or
\[ g^{k+1}(\hat g^{k+1}(p)) \leq g^{k+1}(\widehat{g^{k+1}}(p)) \parallel p. \]
But since the speed of light is $1$ and $x$ is a fixed point, we have
\[ g^{k+1}(r) \geq x|[-(25+k),25+k]\times[-(23+4k),22+n+3k] \geq x|[-k,a+k] \times [-k,n+k] \]
and $x|[-k,a+k] \times [-k,n+k] \parallel p$.
But Lemma~\ref{lem:explodev2} applied $k$ times to $y$ implies that $\diff(x, g^k(y))$, and thus $\diff(x, p)$, intersects $[-k,a+k] \times [-k,n+k]$, a contradiction.
Thus we indeed must have $\hat g^{k+1}(p) = \top$.
\end{proof}

\subsection{\Kynnos{}}
\label{sec:kynnos}

Denote by
\[
  Q =
  \begin{array}{cccccc}
    0 & 0 & 1 & 1 & 0 & 1 \\
    0 & 0 & 1 & 0 & 1 & 1 \\
    1 & 1 & 0 & 0 & 0 & 0 \\
    0 & 1 & 0 & 1 & 1 & 0 \\
    1 & 0 & 0 & 1 & 1 & 0 \\
    1 & 1 & 0 & 0 & 0 & 0 \\
  \end{array}
\]
the fundamental domain of \kynnos, and by $x^Q \in \{0,1\}^{\Z^2}$ the associated $6 \times 6$-periodic configuration with $g(x^Q) = x^Q$.
The following lemma states that it contains a self-enforcing patch (it is essentially a more precise stement of Theorem~\ref{thm:SelfEnforcing}).

\begin{lemma}
  \label{lem:SelfEnforcing}
  There is a finite set $D \subset \Z^2$ such that $p = x^Q|D$ satisfies $\hat{g}(p) = p$.
  Furthermore, there is a finite-support configuration $x \in [p]$ with $g(x) = x$.
\end{lemma}

The patch $p$ is shaped like a $22 \times 28$ rectangle with 8 cells missing from each corner.
It is depicted in Figure~\ref{fig:SelfEnforcing}, together with the still life $x$ containing it.
The patch was found by 
simply applying the function of Algorithm~\ref{alg:Patch} to the $70 \times 70$-patches of the agars we found during our searches. 
\Kynnos{} was the first configuration that yielded a nonempty self-enforcing patch, which we then optimized to its current size.
This lemma directly implies Theorem~\ref{thm:SelfEnforcing}, and almost directly Theorem~\ref{thm:NotChainTrans}.

\begin{proof}[Proof of Theorem~\ref{thm:NotChainTrans}]
  Let $y$ be the finite-support configuration obtained by taking $x$ from the previous lemma and adding a glider that is just about to hit the \kynnos{} patch. It can be checked by simulation that the patch can be annihilated this way. Observe that $y$ is in the limit set $\Omega(g)$: simply shoot the glider from infinity. If $\epsilon > 0$ is very small, in any $\epsilon$-chain starting from $y$ we see the patch destroyed. It is impossible to reinstate it, as the existence of a first step in the chain where it appears again contradicts Lemma~\ref{lem:SelfEnforcing}.
\end{proof}


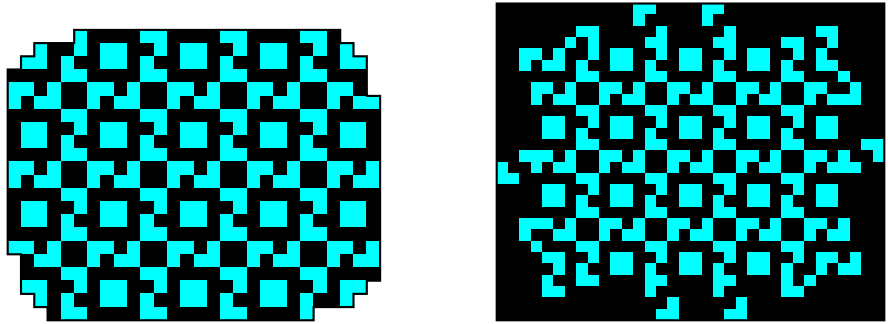
\begin{figure}[htp]
\begin{center}
  \pgfplotstableread{selfenforcingpatch.cvs}{\matrixfile}
  \begin{subfigure}[b]{.45\linewidth}
  \begin{center}\begin{tikzpicture}[scale=0.175]
      \drawit[0]{} 
      \draw[thick,black] (3,-21) -- ++(20,0) -- ++(0,1) -- ++(3,0) -- ++(0,1) -- ++(1,0) -- ++(0,1) -- ++(1,0) -- ++(0,14) -- ++(-1,0) -- ++(0,3) -- ++(-1,0) -- ++(0,1) -- ++(-1,0) -- ++(0,1) -- ++(-20,0) -- ++(0,-1) -- ++(-3,0) -- ++(0,-1) -- ++(-1,0) -- ++(0,-1) -- ++(-1,0) -- ++(0,-14) -- ++(1,0) -- ++(0,-3) -- ++(1,0) -- ++(0,-1) -- ++(1,0) -- cycle;
  \end{tikzpicture}\end{center}
  \vspace{-0.3cm}
  \end{subfigure}
  \pgfplotstableread{selfenforcingstill.cvs}{\matrixfile}
  \begin{subfigure}[b]{.45\linewidth}
  \begin{center}\begin{tikzpicture}[scale=0.15]
  \drawit{}
  \end{tikzpicture}\end{center}
  \vspace{-0.3cm}
  \end{subfigure}
\end{center}
  \caption{A self-enforcing patch of \kynnos{} and a still life containing it. The still life has the minimal number of live cells, 306, of any still life containing the patch. The number was minimized by Oscar Cunningham. \cite{Cu22forum}}
  \label{fig:SelfEnforcing}
\end{figure}

As stated, \kynnos{} can be stabilized from both inside and outside.
Figure~\ref{fig:KynnosHole} shows a still life configuration containing a ``ring'' of \kynnos{} with a hole of 0-cells inside it.
From the figure it is easy to deduce the existence of such rings of arbitrary size and thickness.

Note that if the ring is at least $22$ cells thick, then its interior is completely surrounded by a ring-shaped self-enforcing pattern consisting of translated, rotated and partially overlapping copies of the $28 \times 22$ self-enforcing patch, through which no information can pass without destroying it forever.
If we then replace the empty cells inside the ring with an arbitrary pattern, the resulting finite pattern $P$ occurs in the limit set $\Omega(g)$ if and only if the interior pattern evolves periodically under $g$.
Namely, if the pattern occurs in $\Omega(g)$, then it has an infinite sequence of preimages, each of which must contain the self-enforcing \kynnos{} ring.
The interior has a finite number of possible contents ($2^m$ for an interior of $m$ cells), so it must evolve into a periodic cycle, of which $P$ is part.
From this idea, and some engineering with gadgets found by other researchers and Life enthusiasts, we will obtain Theorems~\ref{thm:rateofgrowth}, \ref{thm:pspacehard} and \ref{thm:nonsofic}.
The first one was essentially proved by Adam Goucher \cite{Go22forum}.
Note the difference between these rings and the \koynnos{} annuli of Section~\ref{sec:koynnos}: the latter force strictly smaller versions of themselves in their preimages, and do not admit nontrivial periodically evolving interiors.

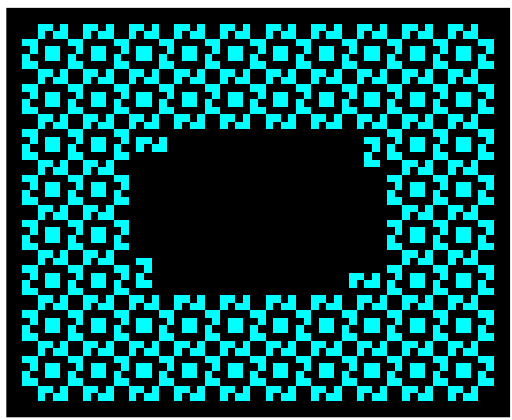
\begin{figure}[htp]
  \begin{center}
    \pgfplotstableread{kynnoshole.cvs}{\matrixfile}
    \begin{tikzpicture}[scale=0.1]
      \drawit{}
    \end{tikzpicture}
  \end{center}
  \caption{A stable ring of \kynnos{}.}
  \label{fig:KynnosHole}
\end{figure}

\begin{proof}[Proof of Theorem~\ref{thm:rateofgrowth}]
  Given integers $k, m \geq 1$ with $k$ odd, we construct a configuration $x \in \{0,1\}^{\Z^2}$ such that the support of $g^n(x)$ is contained in $[0,32k+73] \times [0,46m]$ for all $n \geq 0$, and $g^{48 \cdot 4^{(2k+1)m}}(x)$ is not $g$-periodic.
  When the support of $g^{48 \cdot 4^{(2k+1)m}}(x)$ is surrounded by a \kynnos{} ring of width 22, the resulting pattern has a $48 \cdot 4^{(2k+1)m}$th preimage, but not arbitrarily old preimages.
  If we choose $k = 23 s$ and $m = 16 s$ for some $s \geq 0$, the resulting pattern has size $(736 s + O(1)) \times (736 s + O(1))$, and the chain of preimages has length $48 \cdot 4^{736s^2+16s}$.
  Choosing $s = n/736 - O(1)$ yields the lower bound, and the upper bound is the trivial one (even ignoring the fact we do not modify the boundaries).

  The main components of $x$ are the \emph{period-48 glider gun} \cite{LW-P48G22}, which produces one glider every 48 time steps, and the \emph{quadri-snark} \cite{LW-QS22}, which emits one glider at a 90 degree angle for every 4 gliders it receives.
  The support of $x$ consists of a single period-48 gun aimed at a sequence of $(2k+1) m$ quadri-snarks, each of which receives the gliders the previous one emits.
  They effectively implement a quaternary counter with values in $[0, 4^{(2k+1) m}-1]$.
  The glider emitted by the final quasi-snark will collide with the \kynnos{} ring, ensuring that the pattern right before the impact does not occur in the limit set $\Omega(g)$, but has a chain of preimages of length at least $48 \cdot 4^{(2k+1) m}$.
  An example pattern and a schematic for $m = n = 2$ are given in Figure~\ref{fig:rateofgrowth}.
  It is easy to extrapolate to arbitrary $m, n \geq 1$ from the figure.
\end{proof}

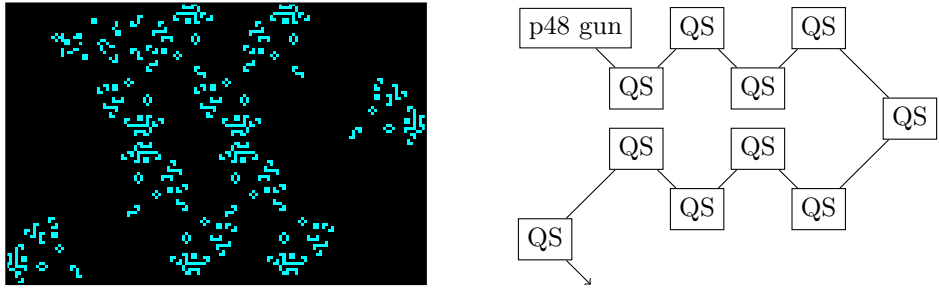
\begin{figure}[htp]
  \begin{center}
    \pgfplotstableread{quadrisnarks.cvs}{\matrixfile}
    \begin{subfigure}[b]{.45\linewidth}
      \begin{center}\begin{tikzpicture}[scale=0.04]
        \drawit{}
      \end{tikzpicture}\end{center}
  \end{subfigure}
  \;\;
  \begin{subfigure}[b]{.45\linewidth}
      \begin{center}\begin{tikzpicture}[scale=0.8]
        \node[draw,rectangle] (gun) at (0,0) {p48 gun};

        \foreach \x/\y/\n in {1/-1/a,2/0/b,3/-1/c,4/0/d,5.5/-1.5/e,4/-3/f,3/-2/g,2/-3/h,1/-2/i,-0.5/-3.5/j}{
          \node[draw,rectangle] (s\n) at (\x,\y) {QS};
        }
        \draw [->] (gun) -- (sa) -- (sb) -- (sc) -- (sd) -- (se) -- (sf) -- (sg) -- (sh) -- (si) -- (sj) -- ++(0.75,-0.75);
      \end{tikzpicture}\end{center}
    \end{subfigure}
  \end{center}
  \caption{A configuration corresponding to $k = m = 2$ in the proof of Theorem~\ref{thm:rateofgrowth}.}
  \label{fig:rateofgrowth}
\end{figure}

To implement more complex Life patterns with desired properties, we use the fact that Life is \emph{intrinsically universal}, that is, capable of simulating all $\Z^2$ cellular automata.
Formally, for any other CA $f : \Sigma^{\Z^2} \to \Sigma^{\Z^2}$, there are numbers $K, T \geq 1$ and an injective function $\tau : \Sigma \to \{0,1\}^{K \times K}$ such that for all configurations $x \in \Sigma^{\Z^2}$ we have $g^T(\tau(x)) = \tau(f(x))$, where $\tau$ is applied cellwise in the natural way.
We use the simulation technique of \cite{DuRo99}, which allow us to easily simulate patterns with \emph{fixed boundary conditions}.
This means that any rectangular pattern $R \in \Sigma^{a \times b}$ can be simulated by a finite-support configuration of $g$ in such a way that simulated cells whose $f$-neighborhood is not completely contained in the rectangle $[0,a-1] \times [0,b-1]$ are forced to retain their value.

\begin{proof}[Proof of Theorem~\ref{thm:pspacehard}]
  Let $L \subset \{0,1\}^*$ be a PSPACE-hard language decidable in linear space, such as TQBF.
  Define a Turing machine $M$ as follows.
  Given input $w \in \{0,1\}^*$, $M$ determines whether $w \in L$ using at most $|w|$ additional tape cells and without modifying $w$.
  If $w \in L$, then it erases the additional tape cells and returns to its initial state, thus looping forever.
  If $w \notin L$, then $M$ stays in a rejecting state forever.
  We simulate $M$ by a cellular automaton $f$ in a standard way: each cell is either empty, or contains a tape symbol and possibly the state of the computation head.

  Next, we simulate the CA $f$ by $g$ as described above.
  Given a word $w \in \{0,1\}^*$, let $P(w)$ be the pattern corresponding to a simulated initial configuration of $M$ on input $w$ with $|w|$ additional tape cells and fixed boundary conditions, surrounded by a \kynnos{} ring.
  If $w \in L$, then $P(w)$ occurs in the limit set $\Omega(g)$, since it can be completed into a $g$-periodic configuration in which the simulated $M$ repeatedly computes $w \in L$.
  If $w \notin L$, then $P(w)$ does not occur in $\Omega(g)$, since the interior of the ring eventually evolves into a simulated configuration with $M$ in a rejecting state, never returning to $P(w)$.
\end{proof}

Extending $P(w)$ by zeroes on all sides (resp.\ repeating it periodically), we obtain that it is PSPACE-hard whether a given finite-support configuration (resp.\ periodic configuration) appears in the limit set.

\begin{proof}[Proof of Theorem~\ref{thm:pspacehardreach}]
  Let $L$ be as in the previous proof, and let $M$ be a Turing machine that, on input $w \in \{0,1\}^*$, decides $w \in L$ using no additional tape cells.
  Then $M$ erases the entire tape and enters an accepting or rejecting state depending on the result of the computation.
  We simulate $M$ by $g$ as in the previous proof.
  Given $w \in \{0,1\}^*$, let $p$ be the pattern corresponding to a tape of $M$ containing $w$ and an initial state, and $q$ the one corresponding to $|w|$ blank tape cells and an accepting state of $M$, both surrounded by a ring of \kynnos{} of the same dimensions.
  Then $q$ is reachable from $p$ if and only if $w \in L$: if $q$ is to be reached, the ring of $p$ must stay intact, enclosing a correct simulation of $M$.
\end{proof}

Of course, again by extending the resulting patterns by 0-cells (resp.\ repeating them periodically), we obtain PSPACE-hardness of reachability between two given finite-support (resp.\ periodic) configurations, i.e.\ given the full descriptions of two configurations $x, y \in \{0,1\}^{\Z^2}$, the question of whether $g^n(x) = y$ for some $n \geq 0$. However, this reachability problem is in fact even $\Sigma^0_1$-complete (resp.\ PSPACE-complete) directly by intrinsic universality. For the case of finite configurations, one needs a variant of intrinsic universality where the zero state of an arbitrary cellular automaton is represented by an all-zero pattern; such a variant was proved in \cite{Goucher0E0P}.

\begin{proof}[Proof of Theorem~\ref{thm:nonsofic}]
  Let $M$ be a two-dimensional Turing machine whose tape alphabet has two distinguished values, denoted $a$ and $b$.
  When $M$ is initialized on a rectangular tape containing only $a$s and $b$s, it repeatedly checks whether its left and right halves are equal, destroying the tape if they are not.
  We again simulate $M$ by a CA $f$, and then $f$ by $g$.
  Then a simulated rectangular tape with the head of $M$ in its initial state, surrounded by a \kynnos{} ring, is in $\Omega(g)$ if and only if the two halves of the tape are equal.

  It was proved in \cite{KaMa13} that for all sofic shifts $X \subset S^{\Z^2}$ there exists an integer $C > 1$ with the following property.
  For all $n \geq 1$ and configurations $x^1, \ldots, x^{C^n} \in X$, there exist $i \neq j$ such that the configuration $y = (x^i|[0,n-1]^2) \sqcup (x^j|\Z^2 \setminus [0,n-1]^2)$ is in $X$.
  Assuming for a contradiction that $\Omega(g)$ is sofic, consider the configurations $x(P) \in \Omega(g)$ for $P \in \{a,b\}^{n \times n}$ that contain a \kynnos{} ring and a simulated tape of $M$ with two identical $P$-halves.
  Based on the above, when $n$ is large enough that $2^{n^2} > C^{K n}$, we can swap the right half of one $x(P)$ with that of another to obtain a configuration $y \in \Omega(g)$ containing a simulated tape of $M$ with unequal halves inside a \kynnos{} ring, a contradiction.
\end{proof}

We remark that a weaker version of Theorem~\ref{thm:rateofgrowth} (where $1/368$ is replaced by a much smaller, or even implicit, constant) could also be proved by intrinsic universality.

\subsection{The marching band}

Let $h = g^2$. Denote by
\[
  R =
  \begin{array}{cccccccc}
  1 & 0 & 0 & 0 & 0 & 0 & 1 & 0 \\
  1 & 1 & 0 & 0 & 1 & 1 & 0 & 0 \\
  1 & 1 & 0 & 0 & 1 & 1 & 0 & 0 \\
  0 & 0 & 1 & 0 & 1 & 0 & 0 & 0 \\
  \end{array}
\]
the fundamental domain of the marching band, and by $x^R \in \{0,1\}^{\Z^2}$ the associated $8 \times 4$-periodic configuration with $h(x^R) = x^R$. The following is proved just like Lemma~\ref{lem:KoynnosForcing}.
Note that the forced region extends outside the original pattern.

\begin{lemma}
  \label{lem:marching}
  Let $x$ be in the spatial orbit of the marching band. Then $\hat h(x|[0,47] \times [0,43]) \geq x|[10, 29] \times [-1, 44]$.
\end{lemma}
 
 \begin{proof}[Proof of Theorem~\ref{thm:nogluing}]
 Let $x$ be in the orbit of $x^R$, and let $p = x|[-a,b]\times[-c,d]$. By the previous lemma, as long as $a+b \geq 48$ and $c+d \geq 44$ we have $\hat h(p) = x|[-(a-10),b-18]\times[-(c+1),d+1]$. Iterating this we get
 \[ \hat h^n(x|[-10n,18n+47]\times[0,43]) \geq x|[0,47]\times[-n,n+43]. \]

 Denote $S = [-10n,18n+47]\times[0,43]$.
 Let $P = x|S$ and $Q = \sigma_{\vec u}(x)|\vec v + S$ for some $\vec u \in \Z^2$ and $\vec v = (0,2n)$.
 Both patterns appear in the limit set of $g$, since they are extracted from a fixed point of $h = g^2$.
Observe that since the domains of $\hat h^n(P)$ and $\hat h^n(Q)$ intersect, we can pick the shift $\vec u$ so that one of the forced bits is different in some position in $\hat h^n(P)$ and $\hat h^n(Q)$, which clearly means $\hat h^n(P \sqcup Q) = \top$. 


Now, $P$ and $Q$ each fit inside a $29n \times 29n$ rectangle (if $n \geq 47$), and the patterns cannot be glued in the limit set with gluing distance at most $2n$, since the glued pattern should have an $n$th $h$-preimage. This gives the statement.
\end{proof}

\section{Chaotic conclusions}
\label{sec:Chaotic}

There are several definitions of topological chaos. We refer the reader to~\cite{Bl09} for a survey. Briefly, a system is called Auslander-Yorke chaotic if it is topologically transitive and is sensitive to initial conditions, and Devaney chaotic if it is Auslander-Yorke chaotic and additionally has dense periodic points. As far as we know, before our results it was open whether Game of Life exhibits these types of chaos on its limit set; the following corollary shows that it does not.

\begin{theorem}
The Game of Life restricted to its limit set is not topologically transitive, and does not have dense periodic points.
\end{theorem}

\begin{proof}
Either of these properties clearly implies chain-nonwanderingness, contradicting Theorem~\ref{thm:NotChainTrans}. 
\end{proof}

Two other standard notions of chaos are Li-Yorke chaos and positive entropy (we omit the definitions). Game of Life exhibits these trivially, since it admits a glider. More generally, intrinsic universality implies that it exhibits any property of spatiotemporal dynamics of cellular automata that is inherited from subsystems of finite-index subactions of the spacetime subshift. Sensitivity in itself is also sometimes considered a notion of chaos. This remains wide open.

\begin{question}
Is Game of Life sensitive to initial conditions?
\end{question}

One can also ask about chaos on ``typical configurations''. For example, take the uniform Bernoulli measure (or some other distribution) as the starting point, and consider the trajectories of random configurations. We can say essentially nothing about this setting.

In our topological dynamical context, a natural way to formalize this problem is through the \emph{generic limit set} as defined in \cite{Mi85}. It is a subset of the phase space of a dynamical system that captures the asymptotic behavior of topologically large subsets of the space.
We omit the exact definition, but for a cellular automaton $f$, this is a nonempty subshift invariant under $f$ \cite{DjGu19}. It follows that the generic limit set is contained in the limit set, and that the language of the generic limit set of $g$ contains the letter $0$ (because the singleton subshift $\{1^{\Z^2}\}$ is not $g$-invariant).

We say a cellular automaton $f$ on $S^{\Z^d}$ is \emph{generically nilpotent} if its generic limit set contains only one configuration, which must then be the all-$s$ configuration for a quiescent state $s \in S$.
This is equivalent to the condition that every finite pattern can be extended into some larger pattern $p$ such that for large enough $n \in \N$, we have $f^n(x)_{\vec 0} = s$ for all $x \in [p]$.
By the previous observation, if Game of Life were generically nilpotent, we would have $s = 0$. We strongly suspect that it is not generically nilpotent, i.e.\ the symbol $1$ occurs in the generic limit set. However, we have been unable to show this.

\begin{question}
\label{q:GenNilpotent}
Is Game of Life generically nilpotent?
\end{question}

Chaos is usually discussed for one-dimensional dynamical system, but we find its standard ingredients, such as topological transitivity and periodic points, quite interesting. We have been unable to resolve most of these.

\begin{question}
Is the limit set of Game of Life topologically transitive as a subshift?
\end{question}

\begin{question}
Does the limit set of Game of Life have dense totally periodic points as a subshift?
\end{question}

\bibliography{golbib}

\end{document}